\documentclass[12pt]{article}
\usepackage[utf8]{inputenc}
\usepackage{graphicx}
\usepackage{color}
\usepackage{times}
\usepackage[hidelinks]{hyperref}
\usepackage{enumerate,latexsym}
\usepackage{latexsym}
\usepackage{amsmath,amssymb}
\usepackage{graphicx}

\usepackage[pagewise]{lineno}

\usepackage{amsmath,amsthm,amsfonts,amssymb}
\usepackage{graphicx}
\usepackage{dsfont}

\makeatletter
\def\namedlabel#1#2{\begingroup
 #2%
 \def\@currentlabel{#2}%
 \phantomsection\label{#1}\endgroup
}
\makeatother

\theoremstyle{plain}
\newtheorem*{theorem*}{Theorem}
\newtheorem*{thmex*}{Theorem~\ref{example}}
\newtheorem*{thmasymp*}{Theorem~\ref{thmAsymp}}
\newtheorem{theorem}{Theorem}[section]

\newtheorem{case}[theorem]{Case}

\newtheorem{lemma}[theorem]{Lemma}

\theoremstyle{definition}

\newcommand{\es}{\emptyset}
\newcommand{\R}{\mathbb{R}}

\newcommand{\Z}{\mathbb{Z}}

\newcommand{\sn}[1]{{\mathbb{S}^{#1}}}
\newcommand{\hn}[1]{{\mathbb{H}^{#1}}}

\newcommand{\ben}{\begin{enumerate}}
\newcommand{\een}{\end{enumerate}}
\newcommand{\wt}{\widetilde}
\newcommand{\g}{\gamma}

\renewcommand{\a}{\alpha}

\newcommand{\cA}{{\mathcal A}}

\newcommand{\cC}{{\mathcal C}}

\newcommand{\D}{\Delta}

\newcommand{\cT}{{\mathcal T}}

\newcommand{\norma}[1]{\Vert #1 \Vert}

\newcommand{\ed}{\end{document}}

\definecolor{rrr}{rgb}{.9,0,.1}

\definecolor{rr}{rgb}{.8,0,.3}

\usepackage{pdfsync}

\begin{document}

\title{A nonexistence result for CMC surfaces in hyperbolic 3-manifolds}

\author{William H. Meeks III
\and Alvaro K.~Ramos\thanks{The authors were partially supported
by CNPq - Brazil, grant no. 400966/2014-0.}.}

\maketitle

\begin{abstract}
We prove that a complete hyperbolic 3-manifold of finite volume does
not admit a properly embedded noncompact surface of finite topology
with constant mean curvature greater than or equal to 1.

\vspace{.15cm}

\vspace{.1cm}

\end{abstract}

\section{Introduction.}

We continue the study of properly immersed surfaces of constant
mean curvature $H$ in hyperbolic 3-manifolds of finite volume
that began with the works of Collin, Hauswirth, Mazet
and Rosenberg \cite{chmr1,chr2} in the minimal case, and was extended to the
$H\in(0,1)$ case by the authors~\cite{meramos1}.

In
this paper we prove:

\begin{theorem}\label{thmtwo}
A complete hyperbolic 3-manifold of finite volume does
not admit a properly embedded noncompact surface of finite topology
with constant mean curvature $H\geq 1$.
\end{theorem}

Theorem~\ref{thmtwo} contrasts with~\cite[Proposition~4.8]{meramos1},
where it is shown that, for any $H\geq 1$ and any noncompact
hyperbolic 3-manifold $N$ of finite volume, there exists a complete,
properly immersed annulus with constant mean curvature $H$. Therefore,
the hypothesis of embeddedness in Theorem~\ref{thmtwo} is  necessary.
Moreover, in~\cite{amr1}, together with Adams, we proved
that for any $H\in[0,1)$ and any surface $S$ of finite negative
Euler characteristic, there exists a
hyperbolic 3-manifold of finite volume with a proper embedding
of $S$ with constant mean curvature $H$.  Furthermore,
there are examples of closed surfaces in hyperbolic 3-manifolds of
finite volume for any $H\geq 1$; namely geodesic spheres and tori and
Klein bottles in its cusp ends.

The work of the first author with Tinaglia~\cite{mt11}
allows one to replace the hypothesis of properness in Theorem~\ref{thmtwo}
by the weaker assumption of completeness, since~\cite{mt11} shows that
any complete, embedded,
finite topology surface of constant mean curvature $H\geq1$ in
a complete hyperbolic 3-manifold is proper.
Finally, there remains the question of whether or not there exist
properly embedded surfaces of infinite topology
and constant mean curvature $H\geq 1$ in  hyperbolic 3-manifolds of finite
volume.

\section{Proof of Theorem~\ref{thmtwo}.}\label{secnonexist}

Theorem~\ref{thmtwo} follows directly from next lemma.

\begin{lemma}\label{propnonexistence}
A complete hyperbolic 3-manifold of finite volume $N$ does
not admit a proper embedding of
$A = \sn1\times[0,\infty)$
with constant mean curvature $H\geq 1$.
\end{lemma}

\begin{proof}
After passing to the oriented two-sheeted cover of $N$, we may assume
without loss of generality that $N$ is orientable.

Arguing by contradiction suppose that $E\subset N$ is the image of
a proper embedding of $A$ as stated in the lemma.
Since $E$ is proper and $N$ is an orientable hyperbolic 3-manifold
of finite volume, there exists some cusp end $\cC$ of $N$
with the following properties:
\begin{enumerate}[1.]
\item $\partial E\cap \cC = \es$.
\item $\partial \cC$ is a flat torus $\cT(0)$ which
intersects $E$ transversely in a finite
collection of pairwise disjoint simple closed curves.
\item $E\cap \cC$ contains a unique noncompact component
$\D$, which is a planar domain.
\end{enumerate}

Since $\D$ is connected and
$\partial \D$ separates $\partial E$ from the end of $E$,
it follows that
$\partial \Delta \subset \cT(0)$
contains a unique simple closed curve $\g\subset \partial \D$
which generates the first homology group $H_1(E)$.
Moreover, any other boundary component
of $\partial \D$ is homotopically trivial in $E$, and hence,
homotopically trivial in $N$. In particular,
$i_*(\pi_1(\D))$ is either trivial or an infinite cyclic subgroup
of $\pi_1(\cC)$, where $i\colon \D\to \cC$ is the inclusion map
and $i_*\colon \pi_1(\D)\to \pi_1(\cC)$ is the induced map on fundamental
groups, after choosing a base point on $\g$.

We let $\Pi\colon \hn3 \to N$ be
the universal cover of $N$ and let $B\subset \hn3$ be a horoball such that
$\Pi\vert_{B}\colon B \to \cC$ is the universal cover of $\cC$.
Using the half-space model for $\hn3$,
we assume, without loss of generality,
that $B$ is the region $B = \{(x,y,z)\in \hn3\mid z\geq 1\}$.

\begin{case}
$i_*\colon \pi_1(\D) \to \pi_1(\cC)$ is trivial.
\end{case}

\noindent
Since $i_*$ is trivial, $i\colon \D\to \cC$ admits a lift
$\wt{i}\colon \D\to B$, whose image $\wt{\D}$
is a properly embedded planar domain
in $B$ with $\partial \wt{\D}\subset\partial B$.
By~\cite[Theorem~10]{chr1} (for $H = 1$)
and~\cite[Theorem~6.9]{kkms1} (for $H > 1$), it follows that
$\wt{\D}$ is asymptotic to a constant mean curvature $H$
annulus $\cA\subset B$, in the sense that
a subend of $\wt{\D}$ is a graph
in exponential normal coordinates over a subend of $\cA$ with graphing
function converging in the $C^2$-norm to zero for any divergent
sequence of points. Moreover, $\cA$ admits
a vertical axis of rotational
symmetry, and there are
three possibilities:
\begin{enumerate}
\item \label{casehoro}
$H = 1$ and $\cA$ is the end of a horosphere;
\item \label{casecat}
$H = 1$ and $\cA$ is the end of an embedded catenoid cousin;
\item \label{casedelaunay}
$H > 1$ and $\cA$ is the end of a Delaunay surface.
\end{enumerate}
Note that the Catenoid cousin in item~\ref{casecat} is
embedded because any end representative of a nonembedded Catenoid
cousin is never contained in a horoball.

In each case, $\cA$ has bounded norm on its second fundamental form
and infinite area; hence, since $\wt{\D}$ is asymptotic to
$\cA$ in the $C^2$-norm, it follows that $\wt{\D}$ also
has bounded norm of its second fundamental form $\norma{A_{\wt{\D}}}$
and infinite area.
Next, we use these properties to get a contradiction.

Since $\wt{\D}$ is a complete, properly embedded surface
with $\partial \wt{\D}\subset \partial B$, then $\wt{\D}$ defines a
mean convex region $M\subset B$ with
$\partial M\setminus\partial B = \wt{\D}$.
Moreover, since $\hn3$ is a homogeneously regular manifold
and $\wt{\D}$ has compact boundary and separates $B$, the bound
on $\norma{A_{\wt{\D}}}$ gives the existence
of a one-sided regular neighborhood in $M$ of radius $\delta>0$,
see~\cite[Lemma~3.1]{mt3}. Since
$\wt{\D}$ has infinite area and $\delta>0$, then $M$ has
infinite volume.

Let $\sigma\colon B\to B$ be a parabolic translation of $\hn3$ that
is a covering transformation of $\Pi$. Since $\Delta$ is embedded,
then $\sigma(\wt{\Delta})\cap \wt{\Delta} = \es$; thus,
either $\sigma(M)\cap M = \es$ or $\sigma(M)\subset M$. Since
$\sigma$ is a translation, the latter is not possible and we obtain
that $\sigma(M)\cap M = \es$. Hence, $\Pi\vert_M\colon M\to \cC$
is injective, which is a contradiction because $\cC$ has finite volume,
and this proves Lemma~\ref{propnonexistence} when $i_*$ is
trivial.

\begin{case} $i_*\colon \pi_1(\D) \to \pi_1(\cC)$ is
nontrivial.
\end{case}

In this case, $i_*(\pi_1(\D))$ is a $\Z$-subgroup of
$\Z\times \Z = \pi_1(\cC)$, generated by $i_*([\g])$.
Let $\wt{\D} \subset B$ be a connected component
of $\Pi^{-1}(\D)$.
Then $\wt{\D}$ is a complete, noncompact, properly embedded planar domain
in $B$ with $\partial \wt{\D} \subset \partial B$; in particular,
it defines a mean convex region $M\subset B$ with
$\partial M\setminus\partial B = \wt{\D}$.

Also, note that
$\wt{\D}$ is invariant under the parabolic covering transformation
$\theta\colon B\to B$ corresponding to $i_*([\g])\in \pi_1(\cC)$.
Since $\partial \D$ is compact, $\partial \wt{\D}$ stays a finite
distance $c>0$ from a line $l\subset \partial B$, invariant under $\theta$.
In order to clarify the next argument, we apply
a rotation around the $z$-axis of $\hn3$ to
assume that $l = \{(0,y,1)\mid y \in \R\}$; hence,
$\partial \wt{\D} \subset \{(x,y,1)\mid x\in (-c,c),\,y \in\R\}$.
Also, after possibly reflecting through the $xz$-plane, we may
assume that $\{(x,y,1)\mid x \geq c \} \subset \partial M$. See
Figure~\ref{figcaseplane}.

\begin{figure}
\begin{center}
\includegraphics[width=0.7\textwidth]{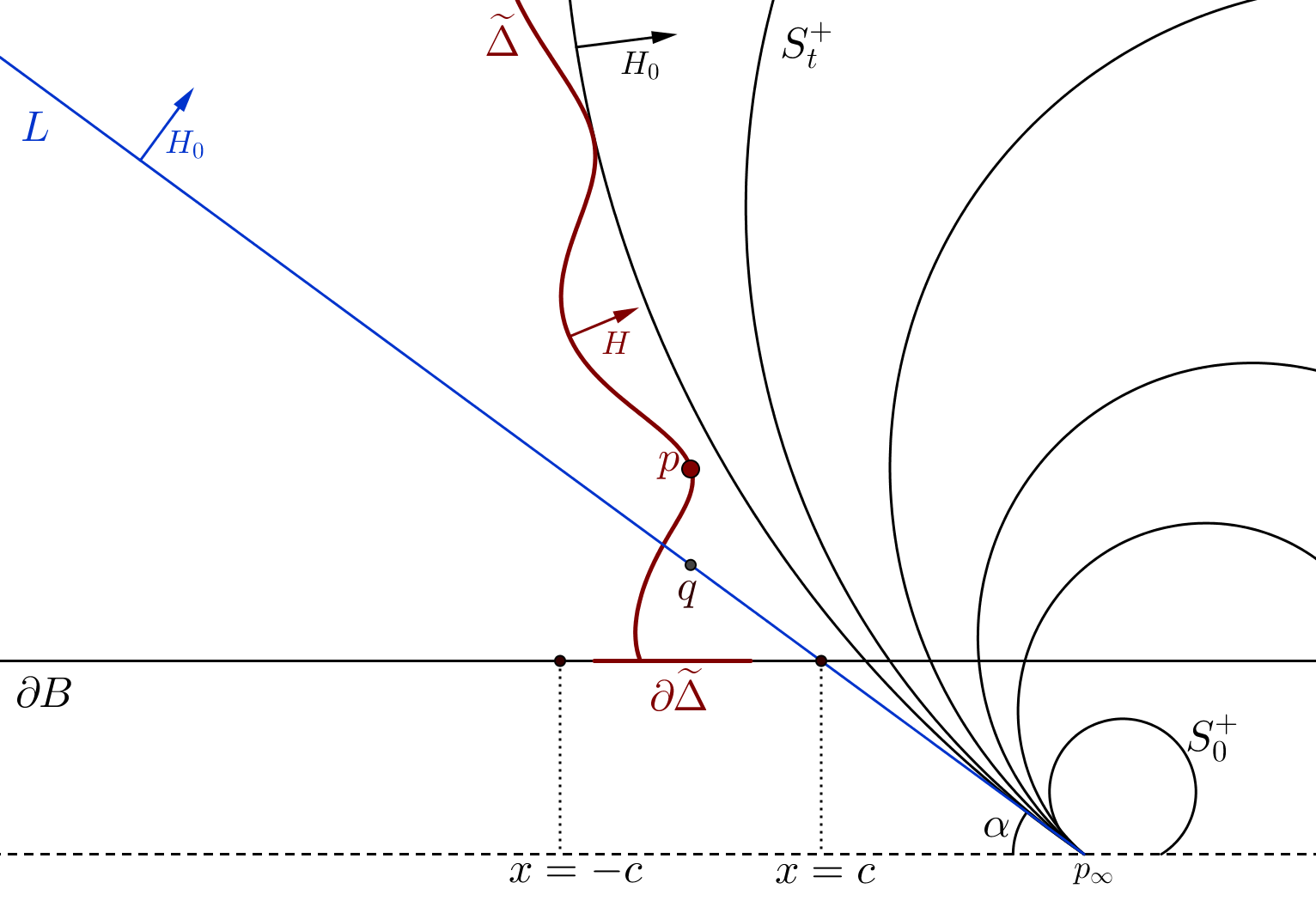}
\caption{\label{figcaseplane} $\wt{\D}$ has constant mean curvature $H\geq 1$
and each hypersphere $S_t^+$ has constant mean curvature $H_0 = \cos(\a)$.
The plane $L$ separates $p$ and $\partial \wt{\D}$,
$p$ lies in the mean convex region of $\hn3$ determined by $L$ and $S_t^+$ converge,
when $t\to \infty$, to $L$.}
\end{center}
\end{figure}

Let $p = (x_1,\,y_1,\,z_1) \in \wt{\D}$ be
such that $z_1 > 1$ and $x_1 \in (-c,\,c)$.
Let $q = (x_1,y_1,\frac{z_1+1}{2})$ and let $L$ be the tilted plane
through $q$ containing the line $\{(c,y,1)\mid y\in\R\}$. Then, $L$ is
an equidistant surface to a totally geodesic vertical plane and,
when oriented with respect to the upper normal vector field,
has constant mean curvature $H_0 = \cos(\a) \in (0,1)$, where $\a$ is the
acute, Euclidean
angle between $L$ and $\{z=0\}$.
Note that $L$ separates $\partial \wt{\D}$ and $p$,
and that $p$ lies in the mean
convex region $U$ defined by $L$ in $\hn3$.

Let $S_0$ be a totally geodesic surface of $\hn3$ such that
$S_0\subset U$, with asymptotic boundary meeting the
asymptotic boundary of $L$ in a single point.
Let $S_0^+$ denote the equidistant surface to $S_0$
with constant mean curvature $H_0$ with respect to the inner
orientation.
Note that we may choose $S_0$
so  that $S_0^+ \cap \partial B = \es$,
as shown in Figure~\ref{figcaseplane}.

Let $M^+\subset \hn3$ denote the mean convex region
defined by $S_0^+$.
Then, there is a product foliation $\{S_t^+\}_{t\geq0}$
of $U\setminus M^+$ such that each surface $S_t^+$ is
equidistant to a totally geodesic surface of $\hn3$ and has
constant mean curvature $H_0$ with respect to the inward orientation;
when $t\to \infty$, the surfaces $S_t^+$ converge to $L$.

Since $p \in \wt{\Delta} \subset B$, then $p\not \in M^+$.
Then, the fact that $p\in U$ implies that
$(\cup_{t\geq0}S_t^+)\cap \wt{\D}\neq \es$.
But because $S_t^+\cap B$ is compact for all $t\geq0$,
there exists a smallest $T>0$ such
that $S_T^+ \cap \wt{\D} \neq \es$.
Since our construction gives that
$\partial \wt{\D}\cap U = \es$,
any point $w$ in $S_T^+ \cap \wt{\D}$
is interior to both $\wt{\D}$ and
$S_{T}^+$, and then $S_T^+$ and $\wt{\D}$ intersect
tangentially at $w$. Finally, since $S_{T}^+\cap B \subset M$,
the mean curvature comparison principle implies that
$H_0 \geq H$, which is a contradiction, and
this proves
Lemma~\ref{propnonexistence}.
\end{proof}

\bibliographystyle{plain}
\bibliography{bill}

\center{William H. Meeks, III at  profmeeks@gmail.com\\
Mathematics Department, University of Massachusetts, Amherst, MA 01003}
\center{Álvaro K. Ramos at alvaro.ramos@ufrgs.br \\
Departmento de Matemática Pura e Aplicada, Universidade Federal do Rio Grande
do Sul, Brazil}

\end{document}